\documentclass[11pt]{amsart}

\usepackage{amssymb,amsmath,amsthm}
\usepackage[all]{xy}
\usepackage{url}

\newtheorem{thm}{Theorem}[section]
\newtheorem{prop}[thm]{Proposition}
\newtheorem{lem}[thm]{Lemma}

\theoremstyle{definition}
\newtheorem{example}[thm]{Example}
\newtheorem{remark}[thm]{Remark}

\numberwithin{equation}{section}

\newcommand{\bbZ}{{\mathbb{Z}}}
\newcommand{\bbP}{{\mathbb{P}}}
\newcommand{\bbG}{{\mathbb{G}}}
\newcommand{\bbC}{{\mathbb{C}}}
\newcommand{\bbQ}{{\mathbb{Q}}}

\newcommand{\bfF}{{\mathbf{F}}}

\newcommand{\Cr}{\operatorname{Cr}}
\newcommand{\tr}{\operatorname{Tr}}
\newcommand{\GL}{\operatorname{GL}}
\newcommand{\SL}{\operatorname{SL}}
\newcommand{\PGL}{\operatorname{PGL}}
\newcommand{\Aut}{\operatorname{Aut}}

\newcommand{\id}{\operatorname{id}}
\newcommand{\Pic}{\operatorname{Pic}}

\newcommand{\bsm}{\left(\begin{smallmatrix}}
\newcommand{\esm}{\end{smallmatrix}\right)}

\newcommand{\la}{\langle}
\newcommand{\ra}{\rangle}

\newcommand{\calR}{\mathcal{R}}

\newcommand{\frakS}{\mathfrak{S}}
\newcommand{\frakA}{\mathfrak{A}}

\newcommand{\beq}{\begin{equation}}
\newcommand{\eeq}{\end{equation}}

\newcommand{\ct}[1]{\!\! \fbox{\rm #1}\! }

\begin{document}

\title{Fixed points of a finite subgroup of the plane Cremona group}

\author{Igor Dolgachev}
\author{Alexander Duncan}
\thanks{The second author was partially supported by
National Science Foundation
RTG grant DMS 0943832.}

\begin{abstract}
We classify all finite subgroups of the plane Cremona group which have a
fixed point.  In other words, we determine all rational surfaces $X$ with
an action of a finite group $G$ such that $X$ is equivariantly
birational to a surface which has a $G$-fixed point.
\end{abstract}

\maketitle

\section{Introduction}

Let $G$ be a finite subgroup of the plane Cremona group, $\Cr(2)$,
the group of birational transformations of the complex projective plane.
We say that $G$ \emph{has a fixed point} if there exists a
smooth rational projective surface $X$ with a faithful $G$-action
$\rho : G \hookrightarrow \Aut(X)$, and a birational map
$\phi : X \dasharrow \bbP^2$ such that
$\phi\circ \rho(G)\circ\phi^{-1} = G$
and $X$ has a $G$-fixed point.
This definition depends only on the conjugacy class of $G$ in $\Cr(2)$.
In this paper we present a classification of conjugacy classes of
subgroups of $\Cr(2)$ with fixed point and, for each class,
we find a representative $G$-surface.

For abelian finite groups acting on smooth proper varieties, the
presence of a fixed point is a birational invariant
(see Proposition~A.2~of~\cite{RY}).
In general, however, this is not true; for example, the exceptional
divisor of a blow up of a fixed point may not have any fixed points.
However, if $f:X\to X'$ is a morphism of $G$-surfaces,
then a fixed point on $X$ maps to a fixed point on $X'$.
Thus, the theory of minimal models of $G$-surfaces tells us that it
suffices to find minimal $G$-surfaces of one of the following two types:

\begin{itemize}
\item \emph{Conic bundles}: there exists a regular map $f:X\to \bbP^1$ such
that the general fiber is isomorphic to $\bbP^1$ and the subgroup
$\Pic(X)^G$ of $G$-invariant invertible sheaves on $X$ is generated over
$\bbQ$ by the canonical class $K_X$ and the class of a fiber of $f$.
\item \emph{del Pezzo $G$-surfaces}: the anticanonical class $-K_X$ is ample
and $\Pic(X)^G$ is generated over $\bbQ$ by $K_X$.
\end{itemize}

An important tool for solving our problem is the classification of
conjugacy classes of finite subgroups of $\Cr(2)$ from \cite{DI}.
Although we use some results from \cite{DI}, many of our proofs do
not directly rely on this work.
In fact, our work led to a discovery of some gaps in the classification and
we use this opportunity to fill these gaps in that paper. 
Note however, that this classification is incomplete in the case of
conic bundles (see also \cite{Tsygankov}).
One may also find an independent classification of abelian subgroups of
$\Cr(2)$ in \cite{Blanc}.

By considering the action on the tangent space of a fixed point,
we see that any finite group $G$ acting on a smooth surface with a fixed
point must be isomorphic to a subgroup of $\GL(2)$.
Also, it is well-known that a cyclic group always has a fixed point on a
rational variety (for example, as was noticed by J.-P. Serre, this
follows from the Lefschetz fixed-point formula applied to the
structure sheaf).
Consequently, we restrict our attention to $G$ not cyclic.

Recall that a del Pezzo surface has degree $d = K_X^2$.
A del Pezzo surface $X$ of degree $4$ can be written by two equations in
$\bbP^4$ defined by diagonal quadrics.
The coordinate hyperplanes cut out 5 genus 1 curves
$E_1$, \ldots, $E_5$ on $X$.
A del Pezzo surface $X$ of degree $3$ is a cubic surface in $\bbP^3$.
An \emph{Eckardt point} on $X$ is a point where three
lines on the surface meet.
A del Pezzo surface $X$ of degree $2$ is a double cover of $\bbP^2$ branched 
over  a smooth plane quartic curve $B$.

\begin{thm} \label{thm:main}
Suppose $G$ is a finite non-cyclic subgroup of the Cremona group
admitting a fixed point.  Then there exists a $G$-surface $X$
realizing a fixed point $p$ of $G$ of one of the following forms:
\begin{enumerate}
\item[\ct{L}] $X$ is $\bbP^2$,
\item[\ct{6}] $X$ is the del Pezzo surface of degree $6$,
\item[\ct{4}] $X$ is a del Pezzo surface of degree $4$ and $p$ lies on
exactly two curves $E_i$, $E_j$, both of which are equianharmonic,
\item[\ct{3}] $X$ is a cubic surface and the tangent space to $p$
contains three Eckardt points,
\item[\ct{2A}] $X$ is a del Pezzo surface of degree $2$ and $p$ lies on
the ramification divisor $R$,
\item[\ct{2B}] $X$ is a del Pezzo surface of degree $2$ and $p$ is the
intersection point of four exceptional curves,
\item[\ct{1}] $X$ is a del Pezzo surface of degree $1$ and $p$ is the base
point of the anti-canonical linear system,
\item[\ct{C}] $X$ is a minimal conic bundle.
\end{enumerate}
\end{thm}

Note that there may be some overlap between these cases as there may be
more than one $G$-surface in an equivalence class.
Occurences of this phenomenon, along with the specific groups that occur
in each case, will be discussed in the sections below.
For the readers convenience, we consolidate those groups
acting on del Pezzo surfaces of degree 2--6 in Table~\ref{tab:DP2to4}.
We use the notations for finite groups employed in \cite{DI} borrowed
from \cite{Atlas}.

\begin{table}[ht]
\begin{center}
\begin{tabular}{|c|c|c|}
\hline
Group & Order & Cases \\
\hline
$2^2$ & $4$ & \ct{2A.1} \\

$\frakS_3$ & $6$ & \ct{6} \\
$\frakS_3$ & $6$ & \ct{3.1} \\

$4 \times 2$ & $8$ & \ct{2B.2}, \ct{2A.3} \\
$D_8$ & $8$ & \ct{2B.1} \\
$Q_8$ & $8$ & \ct{2B.3} \\

$3^2$ & $9$ & \ct{3.3} \\

$6 \times 2$ & $12$ & \ct{2A.2} \\
$\frakS_3 \times 2$ & $12$ & \ct{6} \\
$\frakS_3 \times 2$ & $12$ & \ct{3.2} \\
$3 : 4$ & $12$ & \ct{4} \\

$4^2$ & $16$ & \ct{2B.5} \\
$8 \times 2$ & $16$ & \ct{2A.5} \\
$4.2^2$ & $16$ & \ct{2B.3} \\

$6 \times 3$ & $18$ & \ct{3.3} \\
$\frakS_3 \times 3$ & $18$ & \ct{3.3}(twice)  \\

$12 \times 2$ & $24$ & \ct{2A.4} \\
$2^2 : \frakS_3$ & $24$ & \ct{4} \\
$2 \cdot \frakA_4$ & $24$ & \ct{2B.4} \\

$4 \cdot D_8$ & $32$ & \ct{2B.5} \\

$\frakS_3\times 6$ & $36$ & \ct{3.3} \\

$4 \cdot \frakA_4$ & $48$ & \ct{2B.4} \\
\hline
\end{tabular}
\end{center}
\caption{Non-cyclic subgroups $G$ of $\Cr(2)$ with a fixed point
realized by a minimal del Pezzo $G$-surface of degree 2--6,
but not a minimal conic bundle.}
\label{tab:DP2to4}
\end{table}

We use the opportunity to fill some gaps in the  classification of conjugacy classes in the plane Cremona group from \cite{DI} and we are grateful to Yuri Prokhorov who was the first to observe some of these gaps in his paper \cite{Prok}.

We thank Vladimir Popov who asked the first author about the classification
of finite groups of automorphisms of rational surfaces admitting a fixed
point. 

\section{Preliminaries}

Many of our notations are the same as those in \cite{DI}.
Let $G$ be a finite group.  A \emph{$G$-surface} is a pair $(X,\rho)$
where $X$ is a smooth projective surface and
$\rho : G \hookrightarrow \Aut(X)$ is a faithful $G$-action.
We will often refer to the pair $(X,G)$ or simply $X$ when the context
is clear.
A \emph{morphism} of $G$-surfaces $(X,\rho) \to (X',\rho')$ is a
morphism of the underlying surfaces $f : X \to X'$ such that
$\rho'(G) \circ f = f \circ \rho(G)$.  Similarly, one defines rational
maps, birational maps and birational morphisms of $G$-surfaces.

A $G$-surface $X$ is \emph{minimal} if any birational morphism
$X \to X'$ of $G$-surfaces is an isomorphism.
We say that an action of $G$ on a surface $X$ is a \emph{minimal group
of automorphisms} if the corresponding $G$-surface is minimal.
As in the introduction, minimal $G$-surface are either minimal conic
bundles or minimal del Pezzo $G$-surfaces.

A minimal conic bundle $f:X\to \bbP^1$ is either a minimal ruled
surface with $f$ being one of its rulings, or it has $k \ge 3$
degenerate fibers isomorphic to the union of two $\bbP^1$'s intersecting
transversally at one point.
Recall that a del Pezzo surface $X$ is a smooth projective surface such
that the anticanonical divisor $-K_X$ is ample.  The \emph{degree} of a
del Pezzo surface is $d = K_X^2$, which takes values $1 \le d \le 9$.

We caution the reader that a minimal $G$-surface may be a del Pezzo
surface but not be a minimal del Pezzo $G$-surface!
We shall see an example of such a surface in Section~\ref{sec:DP4}.

With the notable exceptions of $\bbP^2$ and $\bbP^1 \times \bbP^1$, every del Pezzo
surface is a blow up of $\bbP^2$ at $9-d$ points $x_1, \ldots, x_{9-d}$
in general position,
with corresponding exceptional divisors $R_1, \ldots, R_{9-d}$.
Conversely, any set of $9-d$ disjoint $(-1)$-curves can be blown down to
$\bbP^2$; giving rise to a \emph{plane model} of $X$.
Each such choice is called a \emph{geometric marking} and gives rise to
a choice of basis for the orthogonal complement $\calR_X$ of $K_X$
in $\Pic(X)$.

For $d \le 6$, the action of $\Aut(X)$ on $\calR_X$
defines a homomorphism
\beq\label{rho}
\rho:\Aut(X) \to W_{9-d},
\eeq
where $W_n$ denotes the Weyl group of a simple root system of type $E_n$
(by definition, $E_5 = D_5$, $E_4 = A_4$, $E_3 = A_2+A_1$). If $d \le 5$,
then $\rho$ is injective. It follows that in this case any subgroup $G$
of $\Aut(X)$ defines a conjugacy class of $W_n$ which is independent of
a choice of a basis in $\Pic(X)$.

A $G$-surface $X$ is a minimal del Pezzo $G$-surface
if $\Pic(X)^G$ is generated over $\bbQ$ by $K_X$.
A minimal del Pezzo $G$-surface of degree $8$ is isomorphic to
$\bbP^1\times \bbP^1$; the other surface of degree $8$ is never minimal.
Similarly, the surface of degree $7$ is never minimal.

In order to determine whether a del Pezzo surface is minimal, we will use the
following consequence of the Lefschetz fixed-point formula
(as was used in Section~6 of \cite{DI}):

\begin{prop} \label{prop:trace}
Let $X$ be a del Pezzo surface.
If $\sigma$ is an automorphism of $X$, then the trace of
$\sigma^*$ on $\calR_X$ is given by
\[ \tr(\sigma^*|\calR_X) = s - 3 + \sum_{i = 1}^n (2-2g_i) \]
where $s$ is the number of isolated fixed points and
$g_1, \ldots, g_n$ are the genera of the fixed curves.
Moreover, for a finite group $G$, the surface $X$ is $G$-minimal
if and only if
\[
\sum_{\sigma \in G} \tr(\sigma^*|\calR_X) = 0 \ .
\]
\end{prop}

We are classifying $G$-surfaces up to birational equivalence.  It may
happen that two minimal $G$-surfaces are birationally equivalent.
Indeed, we will see in
Sections~\ref{sec:DP9},~\ref{sec:DP8},~\ref{sec:DP6},~and~\ref{sec:DP5}
that all del Pezzo $G$-surfaces of degree $\ge 5$ with a fixed point are
birationally $G$-isomorphic to $\bbP^2$ with a fixed point.
On the other hand, from Section~8~of~\cite{DI} we have that any minimal
del Pezzo $G$-surface of degree $\le 3$ is \emph{rigid}; thus we have
the following.

\begin{lem} \label{lem:Sarkisov}
Every minimal del Pezzo $G$-surface $X$ of degree $\le 3$
is the unique minimal $G$-surface in its
birational $G$-equivalence class.
\end{lem}

The remaining case of degree $4$ is more subtle and will be discussed in
Section~\ref{sec:DP4}.

In Theorems~\ref{thm:cubics},~\ref{thm:DP2A}~and~\ref{thm:DP2B}, we will
describe families of del Pezzo surfaces via normal forms involving
parameters.  For a given family A, there is some collection of groups $G$
which fix a point and for which the $G$-surface is minimal.
For certain special values of these parameters, the
set of possible groups $G$ may be larger and we have a new family B.
We say that A \emph{specializes} to B,
say that B is a \emph{specialization} of A, or write A $\to$ B.
Conversely, we say that A is a \emph{generization} of B.

This language is justified in view of the following:

\begin{prop}
Let $X \to T$ be  a flat family of del Pezzo surfaces of degree $\le 5$
over a base scheme $T$.
For each conjugacy class $C$ of subgroups in $W_{9-d}$,
the set
\[
\{t\in T: \Aut(X_t)\
\textrm{contains a subgroup $G$ representing $C$}\}.
\]
is closed in $T$.
\end{prop}

\begin{proof}
Since the monodromy group of a smooth flat family of del Pezzo surfaces
is a finite subgroup of the Weyl group $W_{9-d}$, after passing to a
certain  finite cover of $T$, we may trivialize the
local coefficient system on $T$ defined by the second cohomology group
of fibers.
Choosing simultaneously a geometric marking
in each fiber, we may define a map from $T$ to the GIT-quotient
$P_2^{9-d}$ of $(\bbP^2)^{9-d}$ by the group $\PGL(3)$.
Since the preimage of a closed set is closed, it suffices to assume that
$T$ is an open subset $U$ of $P_2^{9-d}$ parameterizing point sets whose
blow-up is a del Pezzo surface.
From \cite{DO}, the group $W_{9-d}$ acts biregularly on $U$ via
Cremona transformations and the stabilizer of a point $t \in T$ is equal
to the image of $\Aut(X_t)$ under the homomorphism \eqref{rho}.

Let $a :\Gamma\times V\to V$ be an action of a  finite group $\Gamma$ on
an algebraic variety $V$.  The pre-image $Z$ of the diagonal $\Delta$ of
$V$ under the map $(a ,\textrm{id}):\Gamma\times V\to V\times V$
consists of points $(g,v)$ such that $g\in \Gamma_v$. For any subgroup
$H$ of $\Gamma$, the pre-image of $H$ under the first projection $Z\to
\Gamma$ is a closed subset $W$ of $Z$.
Since $\Gamma$ is finite, the image of $W$ via the second projection
$Z\to V$ is a closed subset of $V$ consisting of points whose stabilizer
contains $H$. Applying this to our situation, where $\Gamma = W_{9-d}$
and $V = U$, we obtain the assertion of the proposition.
\end{proof}

\section{del Pezzo surfaces of degree $9$}
\label{sec:DP9}

In this case, $X \cong \bbP^2$ and we will classify finite
subgroups of $\Aut(\bbP^2)\cong \PGL(3)$ that have a fixed point.

Let $G$ be a subgroup of $\PGL(3)$ and let $\widetilde{G}$ be a preimage
in $\GL(3)$.  We have a three dimensional representation $\rho$ of
$\widetilde{G}$.  The existence of a $G$-fixed point on $X$ is
equivalent to the existence of a $1$-dimensional subrepresentation
$\chi$ of $\rho$.
It follows that any finite group of projective transformation has either
no fixed points, or one fixed point, or three isolated fixed points, or
a line of fixed points plus an isolated fixed point.

\begin{thm}[Case \ct{L}]
\label{thm:L}
Conjugacy classes of finite subgroups of $\Aut(\bbP^2)$
with one isolated  fixed point are in a natural bijection with
conjugacy classes of finite subgroups of $\GL(2)$
with a fixed point.
\end{thm}

\begin{proof}
We choose coordinates $x,y,z$
such that the fixed point is $p_0 = (0,0,1)$. A projective
transformation $g$ fixing this point can be uniquely represented by a
transformation $(x:y:z)\mapsto (ax+by:cx+dy:z)$, where
$\tilde{g} =
\left(\begin{smallmatrix}a&b\\ c&d\end{smallmatrix}\right)\in \GL(2)$. Any
conjugate $g' = h^{-1}gh$ must fix the point $p_0$ (here we use the
assumption on the set of fixed points). Hence $\tilde{h}$ conjugates
$\tilde{g'}$ and $\tilde{g}$.  The converse is also true.
\end{proof}

If $G$ has three isolated fixed points, then $G$ is an abelian
group conjugate to a subgroup of transformations
$(x:y:z)\mapsto (ax:by:cz)$.
Finally, if $G$ has a line of fixed points, then $G$ is a
cyclic group.

\section{del Pezzo surfaces of degree $8$}
\label{sec:DP8}

There are two isomorphism classes of del Pezzo surfaces of degree 8.
One is isomorphic to the blow-up of one point, hence it is not minimal.
The other one is isomorphic to  $X = \bbP^1\times \bbP^1$. So we will 
study subgroups of $\bbP^1\times \bbP^1$.

Assume $G$ has a fixed point $p$. Let
$\ell_1,\ell_2$ be the two fibers passing through $p$. Their union is
$G$-invariant.
The group $G$ contains a subgroup $G'$ of index 1 or 2 such that each
ruling $\pi_i:X\to \bbP^1$ is invariant.

Since each ruling $\pi_i$ is $G'$-equivariant, there must be a $G'$-fixed
point on each image $\pi_i(X) \cong \bbP^1$.  Note that any group of
automorphisms which fixes one point on $\bbP^1$ must fix another.  Thus
there exists another pair of lines $\ell'_1, \ell'_2$ whose intersection
is another $G$-fixed point $p'$ on $X$.

Blowing up $p$, the strict transforms of $\ell_1$ and $\ell_2$ both
become exceptional curves.  Since they do not intersect and their union
is $G$-invariant, we may blow them down $G$-equivariantly to $X'$.
The variety $X'$ is isomorphic to $\bbP^2$ and has a $G$-fixed point since
the birational map $X \to X'$ is defined at $p'$.  Thus, we have the
following:

\begin{thm}
If $G$ has a fixed point on $X\cong \bbP^1 \times \bbP^1$
then $X$ is $G$-birationally equivalent to $\bbP^2$.
\end{thm}

\section{del Pezzo surfaces of degree $6$}
\label{sec:DP6}

The surface $X$ is isomorphic the blow-up of three non-collinear points
$x_1,x_2,x_3$ in the plane.  The strict transforms of the lines
$\ell_{12}, \ell_{13}, \ell_{23}$ through each pair of points
are $(-1)$-curves on $X$.
Along with the exceptional divisors sitting above each point $x_i$,
these form a hexagon of $(-1)$-curves.

Let $p$ be a fixed point of $G$.  If $p$ is on the hexagon, then it
must be one of its vertices since otherwise the side of the hexagon
containing $p$ is $G$-invariant and hence can be equivariantly blown
down.  But, if $p$ is a vertex, then the opposite vertex is also
fixed, and there will be two skew lines that are left invariant and
can be equivariantly blown down. This contradicts the minimality
assumption. Thus $p$ is not on the hexagon.

Since $p$ is not on the hexagon, it must be the preimage of a point $x_0$
in the plane. The blow-up of $X$ at $p$ is a del Pezzo surface of degree 5.
It contains 10 lines, six of them are the preimages of the sides of the
hexagon, three of them are the preimages $\ell_1,\ell_2,\ell_3$  of the
lines in the plane joining $x_0$ with the vertices of the coordinate
triangle.
The last line is the exceptional curve $E(p)$ of the blow-up.
Since the sides of the hexagon and the line $E(p)$ form a $G$-invariant
set of lines, the set of  lines $\ell_1,\ell_2,\ell_3$ is also
$G$-invariant.
The action on this set gives a homomorphism $\rho:G\to \frakS_3$.
If $G$ fixes one line, then we can  blow-down the pair of opposite sides
of the hexagon intersecting this line.
This shows that $(X,G)$ is not minimal.
So, the image of $G$ in $\frakS_3$ is either a cyclic group of order 3,
or the whole $\frakS_3$.
An element in the kernel of $\rho$ fixes all three lines $\ell_i$, and
hence fixes all pairs of opposite sides of the hexagon.
Composing it with the action of the action of the standard Cremona
transformation $s_1$ (see \cite{DI}, p. 487) on $X$ that permutes the
opposite sides, we get the identity.
This shows that $\ker(\rho)$ is either trivial or generated by $s_1$.   

In summary:

\begin{thm} [Case \ct{6}]
Let $G$ be a minimal finite non-cyclic group of automorphisms of a del
Pezzo surface of degree 6 that fixes a point.
Then $G$ is either $\frakS_3$ of order 6 or 
the group $2 \times \frakS_3$ of order 12.
\end{thm}

If we blow up the fixed point $p$ then the lines $\ell_1,\ell_2,\ell_3$
form a $G$-invariant set of skew lines.
Blowing these down, we obtain a $G$-equivariant birational equivalence
from $X$ to $\bbP^1 \times \bbP^1$.  However, the fixed point is lost.
This equivalence is a link of type II (see Section~7~of~\cite{DI}).
From the discussion in Section~8~of~\cite{DI}, we see that
the only other possible minimal del Pezzo or minimal conic bundles
equivariantly birational to $X$ are del Pezzo surfaces of degree $5$.
But such surfaces are only minimal if $G$ contains an element of
order $5$ (see below).  Thus a del Pezzo surface of order $6$ is the
only model for $G$ which has a fixed point.

\section{del Pezzo surfaces of degree $5$}
\label{sec:DP5}

The surface is isomorphic to the blow up of four points $x_1,\ldots,x_4$
in $\bbP^2$, no three of which are collinear.
In this case we know from Theorem 6.4 of \cite{DI} that $\Aut(X) \cong
\frakS_5$.  The 10 exceptional curves along with their intersections are
in bijective correspondence with vertices and lines of the Peterson
graph.  Alternatively, the 10 exceptional curves are in bijection with
pairs of elements of $\{1,2,3,4,5\}$; two curves intersect if and only
if the pairs have no common elements.

The maximal subgroups of $\frakS_5$ are $\frakS_3 \times 2$, $\frakS_4$,
$\frakA_5$ and $5 : 4$.
Note that $\frakS_3 \times 2 \cong \la (123), (45) \ra$ is not
minimal since it fixes the exceptional curve corresponding to $\{4,5\}$.
The group $\frakS_4$ is not minimal since it leaves invariant the $4$ skew
lines  $\{1,5\}, \{2,5\}, \{3,5\}, \{4,5\}$.
The subgraph of the Petersen graph based on the orbit of any cyclic
group of order 5 is a pentagon.
For example, if $\sigma = (12345)$ and the vertex is $\{1,2\}$,
the orbit consists of vertices 
$\{1,2\},\{2,3\},\{3,4\},\{4,5\},\{1,5\}$.
This shows that any group containing an element of order $5$ must be minimal.
Thus, the groups $\frakS_5$, $\frakA_5$ and $5 : 4$ are minimal; however
they do not have $2$-dimensional representations and thus cannot have
fixed points.
Among their subgroups, the only
non-cyclic group not yet considered is $G \cong D_{10}$.

The group $D_{10}$ is minimal and has two fixed points.
To see this, we use the well-known $\frakS_5$-equivariant
isomorphism between a del
Pezzo surface $X$ of degree 5 and the GIT-quotient $P_1^5$ of
$(\bbP^1)^5$ by $\PGL(2)$.
Represented as point sets, the points
\[
p_0 = (1,\epsilon_5,\epsilon_5^2,\epsilon_5^3,\epsilon_5^4)
\textrm{ and }
p_1 = (1,\epsilon_5^3,\epsilon_5,\epsilon_5^4,\epsilon_5^2)
\]
on $X$ are fixed by the group $G = \la \sigma, \tau \ra \cong D_{10}$
where
\[ \sigma = (12345) \textrm{ and } \tau = (25)(34) \ . \]
Indeed, $\sigma(p_i) \equiv p_i$ since it amounts to multiplication by a
constant; while $\tau$ corresponds to $z \mapsto z^{-1}$ on each $\bbP^1$.

While this $G$-surface is minimal, it is birationally equivalent to
$\bbP^2$. 
Note that neither fixed point lies on an exceptional divisor 
since every $G$-orbit of exceptional divisors contains skew divisors.
Considering $X$ as the blow-up of four points in $\bbP^2$, the linear
system of cubic curves through the four points and a double point at the
image of $p_0$ in the plane is of dimension 2.
Thus we have an equivariant birational map from $X$ to $\bbP^2$
which maps $p_1$ to a fixed point.  We conclude:

\begin{thm}
Suppose $(X,G)$ is a minimal del Pezzo surface of degree $5$
with a fixed point and $G$ non-cyclic.   Then $G \cong D_{10}$
and $X$ is $G$-birational to $\bbP^2$ with a fixed point.
\end{thm}

\section{del Pezzo surfaces of degree $4$}
\label{sec:DP4}

We recall several facts from Section~6.4~of~\cite{DI}.
Any quartic del Pezzo surface $X$ is isomorphic to a smooth surface in
$\bbP^4$ given by the  equations
\[
\sum_{i=1}^5 t_i^2 = \sum_{i=1}^5 a_it_i^2 = 0,
\]
where $a_i\ne a_j$ whenever $i\ne j$.

The natural representation of $\Aut(X)$ on the Picard group of $X$
defines an isomorphism $\rho$ of $\Aut(X)$ onto a subgroup of the Weyl
group $W(D_5) \cong 2^4 : \frakS_5$. The normal subgroup $2^4$
is always in the image of $\rho$ and acts on $X$ by multiplying 
an even number of coordinates by $-1$.
The image of $\Aut(X)$ in $\frakS_5$ could be one
of the following groups: $1, 2, \frakS_3, 4$, and $D_{10}$. 

Each element of $2^4$ is represented by a subset $A$ of $\{ 1, 2, 3, 4,
5 \}$ corresponding to the indices of the coordinates $t_i$ that are
multiplied by $-1$.
Since $\Aut(X)$ acts on the projective space $\bbP^4$, we may identify
each subset of $A$ with its complement.
Thus, it suffices to assume that the cardinality of $A$ of a non-trivial
element is equal to $1$ or $2$.
The corresponding involution $\iota_A$ is called \emph{of the first kind} or
\emph{of the second kind}, accordingly.
 
We denote by $E_k$ the genus 1 curve cut out by the hyperplane section
$t_k = 0$. The group $\Aut(X)$ acts on the set of such curves with
kernel of the action equal to $2^4$.
The fixed point set on $X$ of each $\iota_k$ is precisely the corresponding
genus 1 curve $E_k$.

We now discuss how to see the action of $W(D_5)$ on the exceptional
divisors of $X$ and its connection to the plane model.
Recall that $X$ is isomorphic to the blow-up of 5 points
$x_1,\ldots,x_5$ in the projective plane.
We label the 16 exceptional divisors of $X$: let $R_1$, \ldots, $R_5$ be
the exceptional curves corresponding to the points $x_i$, let $R_{ij}$
be the strict transforms of the lines $\overline{x_i,x_j}$,
and let $R_0$ be the strict transform of the conic through the points
$x_1$, \ldots, $x_5$.
Each geometric marking corresponds to a choice of the 5 disjoint lines
$R_i$.
There are $2^4$ such subsets and the Weyl group $W(D_5)$ has $2^4$
conjugate subgroups isomorphic to $\frakS_5$; each of them leaves
invariant the set of the divisor classes of 5 disjoint lines.

Each of the involutions $\iota_k$ is given by a de Jonqui\`eres
involution of the plane model with center at the point $x_k$
(see Section~2.3~of~\cite{DI}).
The involution is given by the linear system of cubics through the
points $x_i, i\ne k,$ and a singular point at $x_k$.
The image of $E_k$ is the unique plane cubic curve that passes through
the points $x_1,\ldots,x_5$ with tangent direction at each point $x_j,
j\ne k,$ equal to the line $\overline{x_j,x_k}$.
The de Jonqui\`eres involution preserves the
pencil of lines through the point $x_k$.

It follows from the
construction of de Jonqui\`eres involutions that $\iota_k$ interchanges
$R_i$ with $R_{ik}$, and $R_k$ with $R_0$. The remaining set of 6
lines $R_{ij}$, where $i,j\ne k,$ consist of three orbits of pairs of
intersecting lines.
Note that, even though no orbits of $(-1)$-curves can be blown down,
the subgroup generated by $\iota_k$ does \emph{not} give $X$ the
structure of a $G$-minimal del Pezzo surface.
However, $X$ is $G$-minimal considered as a conic bundle defined by the
pencil of conics given by the proper inverse transforms of the lines
through $x_k$.

It follows that the involution $\iota_{kl} = \iota_k\circ \iota_l$
interchanges the disjoint lines $R_0$ and $R_{kl}$; thus, it does not
act minimally.
The fixed points of $\iota_{kl}$ are precisely the four intersection
points of the two genus 1 curves $E_k$ and $E_l$.
The only minimal subgroups of $2^4$ with fixed points are those that
contain exactly two involutions of the first kind.

In Section~8~of~\cite{DI}, it is shown that any minimal del Pezzo
$G$-surface of degree $4$ with a fixed point is equivariantly
birationally equivalent to a $G$-minimal conic bundle.
However, the conic bundle may not have a fixed point.
We clarify the situation as follows:

\begin{lem}
Suppose $X$ is a minimal del Pezzo $G$-surface.
\begin{enumerate}
\item If $G$ has more than one fixed point or $G$ is abelian,
then $X$ is birationally equivalent to a minimal conic bundle
with a fixed point.
\item If $G$ has exactly one fixed point and $G$ is non-abelian,
then $X$ is not birationally equivalent to a minimal conic bundle
or non-isomorphic minimal del Pezzo surface with a fixed point.
\end{enumerate}
\end{lem}

\begin{proof}
Let $p$ be a $G$-fixed point on $X$.
Blowing up the point $p$ we obtain a
weak del Pezzo surface $X'$ of degree 3 with $\Pic(X')^G \cong \bbZ^2$.
The linear system $|-K_{X'}-R|$, where $R$ is the exceptional curve of
the blow-up, defines on $X'$ a structure of a $G$-minimal conic bundle.

If $X$ has more than one fixed point, then $X'$ also has a fixed point.
Also, if $G$ is abelian then the induced action on $R \cong \bbP^1$ is
cyclic and thus $X'$ again has a fixed point.

However, if $X$ has a unique fixed point and $G$ is non-abelian,
then the new surface $X'$ does
not have a fixed point.  Indeed, the exceptional curve $R$ has an action
of $G$; since $G$ is not abelian the image of its action is not cyclic
and $R \cong \bbP^1$ cannot have a fixed point.

Conceivably, there might be a third minimal $G$-surface $X''$ birational
to $X$ which \emph{does} have a fixed point.
We consult the classification of elementary links in
Section~7.4~of~\cite{DI}.
From $X$, 
the link of Type I to $X'$ as described above is the only link which
changes the isomorphism class of $X$.  The conic bundle $X'$ satisfies
$K_{X'}^2=3$ and the only links which change the isomorphism class are
links of type II.  These are simply compositions of elementary
transformations and cannot introduce new fixed points, nor change the
value of $K_{X'}^2$.
\end{proof}

We now prove the main result of this section. 

\begin{thm}[Case \ct{4}] \label{thm:DP4}
Let $X$ be a minimal del Pezzo $G$-surface of degree $4$.
Suppose $G$ has a fixed point and is not birationally
equivalent to a minimal conic bundle with a fixed point.
Then $X$ is isomorphic to the $G$-surface 
\[
t_1^2 + \epsilon_3 t_2^2 + \epsilon_3^2 t_3^2 + t_4^2 =
t_1^2 + \epsilon_3^2 t_2^2 + \epsilon_3 t_3^2 + t_5^2 = 0,
\quad \epsilon_3 = e^{2\pi i/3},
\]
whose automorphism group (as an ordinary surface)
is generated by $2^4$ along with the transformations
\begin{align*}
g_1:(t_1:t_2:t_3:t_4:t_5) &\mapsto
(t_2:t_3:t_1:\epsilon_3t_4:\epsilon_3^2t_5)\\
g_2:(t_1:t_2:t_3:t_4:t_5) &\mapsto (t_1:t_3:t_2:t_5:t_4)\ .
\end{align*}
The group $G$ is isomorphic to one of the following groups:
\[ 2^2:\frakS_3,\ 3: 4 \]
with the unique fixed point $p=(1:1:1:0:0)$.
\end{thm}
 
\begin{proof}
From the Lemma, it suffices to find $G$-minimal del Pezzo surfaces with
a unique fixed point and $G$ non-abelian.

It is known that any minimal subgroup of $\Aut(X)$ contains a
non-trivial subgroup of $2^4$. Hence, a fixed point $p$ of $G$ must lie
on one of the curves $E_i$.  
Since no
three genus 1 curves $E_1,\ldots,E_5$ have a common point, the group
$G$ contains a subgroup $G'$ of index $\le 2$ that leaves $E_i$
invariant. We may consider $E_i$ as an elliptic curve with the zero
element $p$. Let $A$ be the image of $G'$ in the automorphism group of
the elliptic curve $E_i$. It is known that $A$ is of order 2, 3, 4, or 6.
It has 4, 3,  2, or 1 fixed points, respectively. Thus $A$ must be of
order $6$, hence the order of $G$ is divisible by 3. 

Let $G$ be a group of automorphisms of $X$ of order divisible by 3. It
is known that $X$ is isomorphic to the surface from the assertion of the
theorem. Also, the automorphism group of $X$ is generated by involutions
$\iota_A$ and the subgroup $H = \la g_1, g_2 \ra \cong \frakS_3$.
We fix a plane
model of $X$ as above to assume that $g_1$ acts on $\Pic(X)$ by
permuting cyclically the classes of the exceptional curves $R_1,R_2,R_3$
and  fixing the curves $R_4,R_5$. The element $g_2$ acts by switching
$R_2,R_3$ and $R_4,R_5$.

There are four subgroups of order $3$ in $\Aut(X)$:
\[
\la g_1\iota_{12} \ra,\
\la g_1\iota_{13} \ra,\
\la g_1\iota_{23} \ra,\
\la g_1 \ra
\]
but they are all conjugate.  We may assume without loss of generality
that $g_1$ is in $G$.

Let $K$ be the kernel of the homomorphism $G \to \frakS_3$ and $\bar{G}$
be the image of this homomorphism.
We enumerate all the possible subgroups $K$ of rank $\le 2$
which are invariant under $g_1$:
\[
\la \iota_4 \ra,\
\la \iota_5 \ra,\
\la \iota_{45} \ra,\
\la \iota_4, \iota_5 \ra,\
\la \iota_{12}, \iota_{23} \ra .
\]
Note that $\la \iota_{12}, \iota_{23} \ra$ does not fix a point and can
be eliminated.
The remaining groups are fixed pointwise by $g_1$.
Thus, if $\bar{G}$ is cyclic of order $3$ then $G$ is abelian
and can be eliminated.
It remains to consider $\bar{G} \simeq S_3$.
In this case, only the subgroups $\la \iota_{45} \ra$ and
$\la \iota_4, \iota_5 \ra$ are invariant under $g_2$;
so these are the only possibilities for $K$.

Consider the set $\Gamma \subset 2^4$ of all elements $\iota_A$
such that $g_2\iota_A$ is in $G$.
Note that $g_3g_2\iota_Ag_3=g_2\iota_{(132)A}$ and
$(g_2\iota_A)^{-1}=g_2\iota_{(12)(45)A}$ are also in $G$.
Also, we note that $(g_2\iota_A)^{-1}(g_2\iota_B)=\iota_A\iota_B$.
Thus $\Gamma$ is an $\frakS_3$-invariant set
such that the product of any two elements in $\Gamma$ is in $K$.
We conclude that $\Gamma$ contains only $\id$, $\iota_4$, $\iota_5$ and
$\iota_{45}$.

Thus the only possibilities for $G$ are
\begin{align*}
\la g_2, g_3, \iota_{45} \ra &\cong 2 \times \frakS_3\\
\la g_2, g_3, \iota_4 \ra &\cong 2^2 : \frakS_3\\
\la g_2\iota_4, g_3 \ra &\cong 3 : 4 \ .
\end{align*}
All of these leave fixed the point $(1:1:1:0:0)$.
Appealing to Proposition~\ref{prop:trace},
we see that $2 \times \frakS_3$ is not minimal while the other two
groups are minimal.
\end{proof}

\begin{remark} \label{rem:Prok}
As was first noticed by Yuri Prokhorov (see \cite{Prok}), the groups
$3 : 4$ and $2^2 : \frakS_3$ above were missing from the
classification in \cite{DI}.
We found additional missing groups isomorphic to
$2\times D_8$, $M_{16}$, $2^3:\frakS_3$, and $L_{16}:3$;
as well as a second copy of $L_{16}$
which is not conjugate to existing group in the list.
In addition, the group of order $32$ identified as $2^2:8$
should instead be $2^3:4$.
Here $L_{16}$ and $M_{16}$
are certain groups of order 16 whose structure is described in Table  3
from \cite{DI}.
One finds the corrected statements and the corrected proofs in a version
of the paper at \url{http://www.math.lsa.umich.edu/~idolga/papers.html}.
\end{remark}

\section{del Pezzo surfaces of degree $3$}

Recall that a del Pezzo surface of degree $3$ is a smooth cubic surface
in $\bbP^3$.  Here we prove the following:

\begin{thm}[Case \ct{3}]
\label{thm:cubics}
Suppose $G$ is a non-cyclic group and
$X$ is a minimal cubic $G$-surface with a fixed point $p$.
Then $X$ is equivariantly projectively equivalent to the surface in
$\bbP^3$ cut out by
\[
F = t_0^3 + t_1^3 + t_2^3 + t_4^3 + t_0t_1(a  t_2 + b   t_3)
\]
with fixed point $p = (0:0:1:-1)$,
where $a$ and $b$ are parameters.
The tangent plane at $p$ contains three Eckardt points.
The different possibilities are given in the following table
\begin{center}
\begin{tabular}{|l|l|l|l|l|}
\hline
Name & Possible $G$ & Parameters &  Surface type from \cite{DI} 
\\
\hline
\ct{3.1} & $\frakS_3$ & & I--VI, VIII, V \\
\hline
\ct{3.2} & $\frakS_3 \times 2$ & $a  = b  $ &
I, II, VI \\
\hline
\ct{3.3} & $\frakS_3\times 6$,
$\frakS_3 \times 3$ (twice), & $a =b  =0$ & I \\
 &
$6 \times 3$, $3 \times 3$ & & \\
\hline
\end{tabular}
\end{center}
which have specializations
\[
\xymatrix{
\ct{3.1} \ar[r] &
\ct{3.2} \ar[r] &
\ct{3.3}
} \ .
\]
Note that we do not list those $G$ which already
occur in generizations.
\end{thm}

\begin{proof}
We begin by considering the case \ct{3.3}.
Here $X$ is  the Fermat cubic surface
\[X: t_0^3 + t_1^3 + t_2^3 + t_3^3 = 0. \]
The automorphism group of $X$ is $3^3 : \frakS_4$
of order $648$ (see~\cite{DI} or~\cite{CAG}).
The surface $X$ has $18$ Eckardt points; one of which is
$p=(0:0:1:-1)$ and all the others are obtained from $p$ by
automorphisms.  The stabilizer $\Aut(X,p)$ is isomorphic to
$\frakS_3 \times 6$ of order $36$.
We will show that every case specializes to this case.

Now, let $X$ be general as in the theorem.
Since $G$ is minimal, the cardinality of any orbit on the 27 lines must
be divisible by $3$ (if the sum of $k$ lines is linearly equivalent to
$mK_X$, then $k= 3m$).   Thus, $G$ has order divisible by $3$.

Let $g$ be an element of order $3$ in $G$.
From Table~9.5~of~\cite{CAG}, up to projective equivalence, we have
three different options for the action of $g$ on $\bbP^3$:
\begin{align*}
(3A) \colon & g(t_0:t_1:t_2:t_3) = (\epsilon_3 t_0:t_1:t_2:t_3)\\
(3C) \colon & g(t_0:t_1:t_2:t_3) = (\epsilon_3 t_0:\epsilon_3 t_1:t_2:t_3)\\
(3D) \colon & g(t_0:t_1:t_2:t_3) = (\epsilon_3 t_0:\epsilon_3^2 t_1:t_2:t_3)
\end{align*}
where $\epsilon_3$ is a primitive $3$rd root of unity.

If we assume that $g$ is of the class $(3C)$ then $X$ must be the
Fermat cubic (see Section 9.5.1~of~\cite{CAG}).  The fixed points are all
Eckardt points and so we are in case \ct{3.3}.

Now, we assume that $g$ is of class $(3D)$.
As in Section 9.5.1~of~\cite{CAG}), up to a projective
change of coordinates, the surface $X$ is one of the surfaces stated in
the theorem.  Set $\ell_1:t_2=t_3=0$ and $\ell_2:t_0=t_1=0$.
Note that $\ell_1$ and $\ell_2$ are canonically defined given $g$.
The three points $\ell_2
\cap X$ are the only fixed points of $g$ on $X$;
they are of the form $(0:0:1:-a)$ where $a^3=1$.
Without loss of generality we may take $p=(0:0:1:-1)$.

There is always an involution $\sigma$ which interchanges $t_0$ and $t_1$.
Thus $\frakS_3$ always acts on $X$ fixing $p$.
The line $\ell_1$ is stable under $\frakS_3$ and the three points
$X \cap \ell_1$ are all Eckardt points by Proposition~9.1.27 of~\cite{CAG}.

Consider the tangent space $T_pX \subset \bbP^3$.
One checks that $T_pX$ contains $\ell_1$.
The intersection
$C = T_pX \cap X$ is either a nodal cubic or three concurrent lines.
There is a faithful action of $G$ on $T_pX$ which must leave $C$
invariant.

In the case of $C$ a nodal cubic, we see that
$G \subset \bbG_m : 2$.  Since, $G$ contains an element of order $3$ and
the three points in $\ell_1 \cap C$ must be $G$-invariant.  We see that
$G$ is isomorphic to $\frakS_3$ and we are in case \ct{3.1}.

When $C$ is three concurrent lines, the point $p$ is an Eckardt point
and we have $\frakS_3 \times 2 \subset \Aut(X,p)$ by 
Proposition~9.1.26 of~\cite{CAG}.  The automorphism group of $3$
concurrent lines in $\bbP^2$ is $\bbG_m \times \frakS_3$.  The polar $P$
of $p$ in $X$ is a union of two planes, the tangent plane $t_2+t_3=0$
and the plane $t_2-t_3=0$.  Since both these planes and the line
$\ell_1$ must be $G$-invariant, the only other automorphisms fixing $p$
must be of the form
$(t_0:t_1:t_2:t_3) \mapsto (t_0:t_1:\lambda t_2:\lambda t_3)$
where $\lambda$ is in $\bbC^\times$.
This $\lambda$ can be non-trivial only in the case where $X$ is the
Fermat cubic (and we are then in case \ct{3.3}).
If $\lambda$ is forced to be trivial, then we are in case \ct{3.2}.

Finally, if we assume that $g$ is of the class $(3A)$, then $X$
is cyclic cubic surface
\[
X:t_0^3+F(t_1,t_2,t_3) = 0.
\]
It is a triple cover of $\bbP^2$ ramified at the smooth genus 1 curve
cut out by the plane $t_0= 0$.
The Hessian quartic surface is a union of the plane $t_0=0$ and a cone
over the Hessian cubic curve $H$ associated to $\bbP^2$.
If $X$ has an additional cyclic structure, then the curve $H$ is a
union of 3 concurrent lines and $X$ must be isomorphic to the Fermat cubic
(see Lemma~3.2.4~of~\cite{CAG}).
Since the Fermat cubic was already considered above,
we may assume the cyclic structure is unique and, thus,
$G$ leaves invariant a genus 1 curve containing the fixed point $p$.
This means that $G$ is a central 
extension of a cyclic group $H$ by $3$.
The group $G$ is thus cyclic unless $H$ has order divisible by $3$.
In this latter case, $G$ contains a subgroup isomorphic to $C_3^2$.
This means that $G$ must contain an element of order $3$ whose class is
not of the form $(3A)$ and thus was already discussed above.

It remains to determine which subgroups $G$ of $\Aut(X,p)$ are minimal.
It suffices to consider only \ct{3.3} since the others are generizations
of this case.
First, we list all non-cyclic subgroups of $\frakS_3 \times 6$
up to conjugacy:
\[2^2,\
\frakS_3\ \mathrm{(twice)}, \
3^2, 6 \times 2,\
\frakS_3 \times 2,\
6 \times 3,\
\frakS_3 \times 3\ \mathrm{(twice)},\
\frakS_3 \times 6.
\]
We compute the traces on the space $\calR_X$ using
Proposition~\ref{prop:trace}
as in Table~9.4~of~\cite{CAG}:
\begin{center}
\tiny
\begin{tabular}{|cccc|c|}
\hline
\multicolumn{4}{|c|}{Eigenvalues on $\bbP^4$} & $\tr(\cdot|\calR_X)$ \\
\hline
1 & 1 & 1 & 1 & 6\\
1 & 1 & 1 & $-1$ & -2\\
1 & 1 & $-1$ & $-1$ & 2 \\
1 & 1 & 1 & $\epsilon$ & -3 \\
1 & 1 & $\epsilon$ & $\epsilon$ & 3 \\
1 & 1 & $\epsilon$ & $\epsilon^2$ & 0 \\
1 & 1 & $\epsilon$ & $-\epsilon$ & 1 \\
1 & -1 & $\epsilon$ & $\epsilon$ & 1 \\
1 & -1 & $\epsilon$ & $-\epsilon$ & -1 \\
1 & -1 & $\epsilon$ & $\epsilon^2$ & -2 \\
\hline
\end{tabular}
\end{center}
where $\epsilon$ is a primitive third root of unity.

Note that the subgroup generated by an element with eigenvalues
$1,1,1,\epsilon$ give traces which sum to $0$, thus any group containing
this element is minimal.
Thus, $3 \times 3$, $6 \times 3$, both classes of $\frakS_3 \times 3$, and
$\frakS_3 \times 6$ are minimal groups.
We remark that the eigenvalues of the involutions in the two different
classes of $\frakS_3 \times 3$ are different; thus the two conjugacy
classes are distinct in $\Cr(2)$.

Additionally, the group $\frakS_3$ generated by elements with
eigenvalues $1,1,\epsilon, \epsilon^2$ and $1,1,1,-1$ gives traces which sum
to $0$.
Thus, $\frakS_3$ and $\frakS_3 \times 2$ are minimal.
It remains to establish that $6 \times 2$, $2 \times 2$ and the other
conjugacy class of $\frakS_3$ are not minimal.
The group $6 \times 2$ and the group $\frakS_3$ both have traces which
sum to $12$, so neither is minimal;
the group $2^2$ is a subgroup of $6 \times 2$ and thus, is not minimal.
\end{proof}

\section{del Pezzo surfaces of degree $2$}

Throughout this section, $G$ is a non-cyclic finite group,
$X$ is a minimal del Pezzo $G$-surface of degree $2$, and
$p$ is a $G$-fixed point on $X$.

We recall some features of such surfaces from Section 6.6~\cite{DI}.
Any such surface has an involution $\gamma$ called
the \emph{Geiser involution}.  Its set of fixed points is a smooth curve
$R$ of genus 3.  The quotient by $\gamma$ induces a degree $2$ map
\[ \pi : X \to \bbP^2\]
with branch locus $B \cong R$ a smooth quartic curve.

We may write $X$ as:
\[F(t_0,t_1,t_2)+t_3^2=0 \]
in the weighted projective space $\bbP(1,1,1,2)$, where $F$ is the
degree $4$ form which defines $B$ in $\bbP^2$.  The Geiser involution
is simply the map which takes $t_3$ to $-t_3$.
We have a decomposition $\Aut(X) \cong \Aut(B) \times
\langle \gamma \rangle$.  Note that $\Aut(B)$ is a finite subgroup
of $\PGL(3)$ since $F=0$ is the canonical embedding of $B$.
The possible $\Aut(B)$ can be found in
Theorem~6.5.2~of~\cite{CAG}.

\begin{thm}[Case \ct{2A}]
\label{thm:DP2A}
If $p$ lies on the ramification curve $R$ then the group
$\Aut(X,p)$ is abelian of the form $H \times \langle \gamma
\rangle$ where $H$ is a cyclic subgroup of $\Aut(B)$.
We have the following possibilities  

\begin{center}
\begin{tabular}{|l|l|l|l|}
\hline
Name & Possible $G$ & Surface type from \cite{DI} \\
\hline
\ct{2A.1} & $2^2$ & I--V, VII--X, XII \\
\hline
\ct{2A.2} & $6 \times 2$ & III, VIII \\
\hline
\ct{2A.3} & $4 \times 2$ & II--III, V \\
\hline
\ct{2A.4} & $12 \times 2$ & III \\
\hline
\ct{2A.5} & $8 \times 2$ & II \\
\hline
\end{tabular}
\end{center}
satisfying the specializations
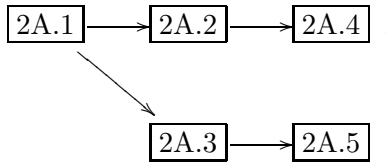
\[
\xymatrix{
\ct{2A.1} \ar[r] \ar[dr] &
\ct{2A.2} \ar[r]&
\ct{2A.4} \\
&
\ct{2A.3} \ar[r] &
\ct{2A.5}
} \ .
\]
Note that we do not list those $G$ which already
occur in generizations.
\end{thm}

\begin{proof}
Since $p$ lies on $R$, $\Aut(X,p)$ contains $\gamma$.
It remains only to classify the possible $H$.
Since $H$ acts faithfully on the tangent space to $R$ at $p$,
we see that $H$ is cyclic.
Since $G$ is not cyclic, $\Aut(X,p)$ is not cyclic.
Thus, the possible $H$ are precisely the maximal cyclic subgroups
of $\Aut(B)$ of even order which fix a point on $B$.
From Lemma~6.5.1~of~\cite{CAG}, we obtain the cases
\ct{2A.1}-\ct{2A.5} above.
Minimality follows since the Geiser involution alone is minimal.
\end{proof}

Now, suppose $G$ does not fix any points on the ramification curve $R$.
Then $p$ is not fixed by $\gamma$ and we may assume that $G$ is an
isomorphic lift of  a subgroup $\bar{G}$ of $\Aut(B)$ fixing a point $q
= \pi(p)$ in $\bbP^2$ not lying on $B$.

A del Pezzo surface has 56 exceptional curves (lines) $E_i$ on which $G$
acts. Any orbit of $G$ on the lines consists of $k$ lines whose sum is
linearly equivalent to a multiple of $K_X$. Since $K_X^2 = 2$, this
implies that $k$ is even. Thus the order of $G$ is even and $G$ contains
an involution $\tilde{\tau}$, a lift of an involution $\tau$ of $\bbP^2$
that leaves $B$ invariant. The  set $(\bbP^2)^\tau$ of fixed points of $\tau$
is equal to $\{q\}$ plus a line $L$ that intersects $B$ at four  fixed
points (counted with multiplicities). The set of fixed points of
$\tilde{\tau}$ is the set containing the two points $p$ and $\gamma(p)$
along with a genus 1 curve $\pi^{-1}(L)$.

We claim that $q$ is the intersection point of four bitangents.
Choose the projective coordinates $(t_0,t_1,t_2)$ in $\bbP^2$ such that
$q = (0:0:1)$ and $L: t_2 = 0$.  Then the equation of $B$ has the form
\beq\label{eq1}
t_2^4+2f_2(t_0,t_1)t_2^2+f_4(t_0,t_1) = 0
\eeq
where the involution $\tau$ acts by $(t_0:t_1:t_2)\mapsto (t_0:t_1:-t_2)$
and $f_2$ and $f_4$ are homogeneous polynomials of degree $2$ and $4$,
respectively.
Note that we can rewrite the equation in the form
\beq
(t_2^2+f_2(t_0,t_1))^2+(f_4(t_0,t_1)-f_2(t_0,t_1)^2) = 0 \ .
\eeq
This shows that each line $bt_0-at_1 = 0$, where $(f_4(a,b)-f_2(a,b)^2)
= 0$, is a bitangent of $B$ passing through the point $q$. Thus $q$ is
the intersection point of four bitangents of $B$ as claimed.

The converse was first proven by Sonya Kowalevski \cite{Kow}.
Although we do not use  this result we give a proof. 

\begin{prop} \label{prop:kow}
Suppose a smooth plane quartic curve $B$ has four bitangents meeting at
a point $q$. Then the exists a projective involution $\tau$ of $\bbP^2$ that
leaves $B$ invariant and has the  point $q\in B$ as an isolated fixed point.
\end{prop}

\begin{proof}
By Proposition 6.1.4 from \cite{CAG}, any three of the bitangent lines
form a syzygetic triad of bitangents, i.e. the corresponding six
tangency points lie on a conic. This implies that all eight tangency
points lie on a conic. Choose coordinates so that $q = (0:0:1)$. Let
$\ell_i:l_i = a_it_0+b_it_1 = 0$ and let 
$C_2(t_0,t_1,t_2) = 0$ be the
equation of the conic $K$ passing through the eight tangency points.
Then the polynomials $C_2^2$ and $l_1\cdots l_4$ define  the same
divisor on $B$, hence the equation of $B$ can be written in the form $F
= C_2^2+l_1l_2l_3l_4 = 0$. Let  $C_2 =
t_2^2+2t_2l(t_0,t_1)+q(t_0,t_1) =
(t_2+l(t_0,t_1))^2+q(t_0,t_1)-l(t_0,t_1)^2 =  0$. After we change again
the coordinates $t_2\mapsto t_2+l(t_0,t_1)$, the equation of $B$ is
reduced to the form \eqref{eq1}. The involution $(t_0:t_1:t_2) \mapsto
(t_0:t_1:-t_2)$ is the projective involution of $B$.
\end{proof}

The involution $\tau$ has four fixed points $(a:b:1)$ on $B$,
where $f_4(a,b) = 0$,
and the quotient $E = B/(\tau)$ is a genus 1 curve with equation
\beq\label{eq2}
z^2+2zf_2(x,y)+f_4(x,y) = 0
\eeq
in the weighted projective space $\bbP(1,1,2)$.

\begin{lem}\label{lem:center}
The involution $\tau$ of $B$ belongs to the center of the group $\bar{G}$.
\end{lem}

\begin{proof}
For any $\sigma\in \bar{G}$, the element $\tau' = \sigma\tau\sigma^{-1}$ fixes
$q$ and leaves invariant the set of the bitangents of $B$ that contain
$q$. Thus it leaves invariant the pencil of lines through $q$. This
shows that $\tau'$ is an involution of $\bbP^2$ with the same isolated
fixed point as $\tau$.  Thus $\tau$ and $\tau'$ must coincide.
\end{proof}

Since the polynomial $f_4$ has four distinct roots, we can choose
projective coordinates $t_0,t_1,t_2$ in the plane to assume that
\[
f_2(t_0,t_1) = at_0^2+bt_0t_1+ct_1^2, \quad
f_4(t_0,t_1) = t_0^4+dt_0^2t_1^2+t_1^4 \ .
\]
The only condition on the coefficients here is $d^2\ne 4$ expressing the
fact that $f_4$ has no multiple roots, or, equivalently, the curve $B$
is nonsingular.

We may assume that $\bar{G}$ acts via its lift to $\GL(3)$ as the
group of matrices of the form
$\bsm
\alpha & \beta & 0\\
\gamma & \delta & 0\\
0 & 0 & 1
\esm$.
Thus the group $\bar{G}$ is naturally identified with a subgroup of $\GL(2)$.
The transformation $\tau$ is defined by the matrix $-I_2$. 

We want to find a list of maximal non-cyclic subgroups of $\GL(2)$,
up to conjugacy, that leave $f_2$ and $f_4$ invariant.
The automorphism $\tau$ is always present.
Let $H$ be the automorphism group of $f_4=0$ viewed as a set of $4$ points
in $\bbP^1$.
Let $K$ be the image of $G$ in $\PGL(2)$; note that $K \subset H$.
Consulting Section~5.5~of~\cite{DI}, we see that for $f_4$ in the
coordinates above, $H$ is either
$2^2$ for general $d$,
$\frakA_4$ for $d^2=-12$,
or $D_8$ for $d=0$.
If $f_2=0$ then the kernel of $G \to K$ is cyclic of order $4$
and the possible maximal $G$ are, respectively,
$4.2^2$, $4.\frakA_4$ and $4.D_8$.

Suppose $f_2 \ne 0$.
The kernel of $G \to K$ is precisely $\la \tau \ra$.
Since $K$ must leave a pair of points invariant,
it is isomorphic to one of $1$, $2$, $3$, $4$ or $2^2$.
We may discount $1$ and $3$ since we consider non-cyclic $G$.
Since $G$ must be a subgroup of $4.2^2$ or $4.D_8$,
all of its elements act by scaling $t_0$ and $t_1$ while possibly
interchanging them.
Thus, all the remaining possibilities arise when $a$, $b$, or $c$ is zero, or
when $a=c$.

Note that if $a=c$ then we may instead assume $b = 0$ via the
the linear change of variables
\[ (t_0,t_1) \mapsto (\delta (t_0-t_1), \delta (t_0+t_1) ) \]
for some $\delta$ satisfying $\delta^4=(2+d)^{-1}$.
Accounting also for the symmetry between $a$ and $c$,
we enumerate the possibilities in Table~\ref{tab:maximalG}.

\begin{table}[ht]
\begin{center}
\begin{tabular}{|l|l|l|}
\hline
$f_4$ & $f_2$ & Maximal $G$ \\
\hline
any $d$ & $a-c,a,c \ne 0$, $b = 0$ & $2^2$ \\
\hline
any $d$ & $a = c\ne 0$, $b=0$ & $D_8$\\
\hline
$d = 0 $ & $a = b = 0, c \ne 0$ & $2\times 4$ \\
\hline
$d \ne 0,d^2 \ne -12$ & 0 & $4.2^2$ \\
\hline
$d^2 = -12$ & 0 & $4.\frakA_4$ \\
\hline
$d = 0$ & 0 & $4.D_8 \cong 4^2:2$ \\
\hline
\end{tabular}
\end{center}
\caption{
Maximal non-cyclic subgroups $G$ of $\GL(2)$
leaving $f_2$ and $f_4$ invariant.}
\label{tab:maximalG}
\end{table}

Our group $G$ is a minimal isomorphic lift of a subgroup $\bar{G}$ of $
\Aut(B)$ as above.  Following Section~6.6~of~\cite{DI},
we say a lift is \emph{even} if the group $G$ in its
representation in $W(E_7)$ is contained in the normal subgroup
$W(E_7)^+$ of index 2, and a lift is \emph{odd} otherwise.

\begin{remark}
The classification of minimal groups of automorphisms of degree 2 del
Pezzo surfaces from \cite{DI} has the following errors.
\begin{enumerate}
\item $\la \gamma \ra$ is missing from all types (except XII).
\item Type XIII is missing completely.
\item $2^2\times \la \gamma \ra$ was omitted in Type I surfaces.
\item An even lift of $Q_8$ was omitted in Types II, III and V.
\item $2 \cdot \frakA_4 \cong Q_8 : 3$ in Type III
(not $D_8 : 3$).
\item $\frakA_4\times \la \gamma \ra$ was omitted in Type IV.
\item $C_3 \times \la \gamma \ra$ was omitted in Type III.
\end{enumerate}
One finds the corrected statements and the corrected proofs in a version
of the paper at \url{http://www.math.lsa.umich.edu/~idolga/papers.html}.
\end{remark}

\begin{thm}[Case \ct{2B}]
\label{thm:DP2B}
Let $G$ be a minimal group with a fixed point $p$ that is not in the
ramification curve $R$, then $B$ is isomorphic to the plane quartic
curve \[F = t_2^4 + t_2^2(at_0^2+ct_1^2) + t_0^4 +
dt_0^2t_1^2+t_1^4 = 0 \] and the fixed point is a lift of $(0:0:1)$.

We have the following cases:
\begin{center}
\begin{tabular}{|l|l|l|l|l|}
\hline
Name & Possible $G$ & Parameters & Surface type from \cite{DI}  \\
\hline
\ct{2B.1} & $D_8$ & {\tiny $a = c \ne 0$} & I-V, VII \\
\hline
\ct{2B.2} & $2 \times 4$
& {\tiny $a  =  d = 0$} & II, III, V  \\
\hline
\ct{2B.3} & $4\cdot 2^2 \cong 2.D_8$,
$Q_8$ & {\tiny $a  =  c = 0$} & II, III, V \\
\hline
\ct{2B.4} & $4 \cdot \frakA_4$,
$2 \cdot \frakA_4$ &
{\tiny $a  =  c = 0$, $d ^2 = -12$} & III \\
\hline
\ct{2B.5} & $4 \cdot D_8$,
$4 \times 4$ & {\tiny $a  =  c = d = 0$}
& II \\
\hline
\end{tabular}
\end{center}
satisfying the specializations
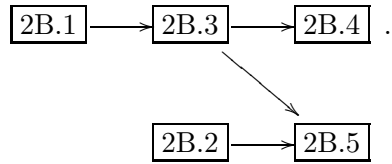
\[
\xymatrix{
\ct{2B.1} \ar[r] &
\ct{2B.3} \ar[r] \ar[dr]&
\ct{2B.4} \\
&
\ct{2B.2} \ar[r] &
\ct{2B.5}
} \ .
\]
Note that we do not list those $G$ which already
occur in generizations.
\end{thm}

\begin{proof}
We may assume that the fixed point is $(0:0:1)$ and that $G$ acts as a
subgroup of $\GL(2)$ on the coordinates $(t_0, t_1)$.
The maximal $G$ and the appropriate parameters can be obtained from
Table~\ref{tab:maximalG}.  It remains only to determine which subgroups
make $X$ minimal.
Our main tool is Proposition~\ref{prop:trace} and the classification
in Table~7~of~\cite{DI} (which was derived using the same method).

Observe that any involution in $\PGL(3)$ fixes an isolated point and a
line.  One lift to $\Aut(X)$ fixes points only on $R$ and thus is
excluded.  The other fixes a pair of points and a genus 1 curve,
thus every involution has trace $-1$ on $\calR_X$.
In particular, $G \cong 2^2$ has sum of traces equal to $4$ and cannot
be minimal.

An element of order $4$ in $G$ with eigenvalues $i,-i$ in $\GL(2)$ fixes
2 points on $X$ and thus has trace $-1$ on $\calR_X$.
From this we conclude that both $D_8$ and $Q_8$ are minimal groups.
In particular, the group $D_8$ in \ct{2B.1} is minimal.

The element of order $4$ with a generator having eigenvalues $1,i$ in
$\GL(2)$ fixes a genus 1 curve on $X$ and thus has trace $-3$.
The group it generates is minimal.
Thus the group $2 \times 4$ from \ct{2B.2} is minimal.

Now we refer to Table~7~of~\cite{DI}.  The cases \ct{2B.1} and \ct{2B.2}
are finished.  For \ct{2B.3}, we note that $Q_8$ appears and that
$4\cdot 2^2$ is minimal since it contains $D_8$.  For \ct{2B.4},
$2\cdot \frakA_4$ and $4\cdot \frakA_4$ contain $Q_8$ and are therefore
minimal.

In the case of \ct{2B.5}, we only need to consider subgroups of a
Sylow $2$-subgroup of $\Aut(X)$ isomorphic to $4 \cdot D_8$.
The group $4^2$ contains $2 \times 4$ and is thus minimal.  The group
$4 \cdot D_8$ is similarly minimal.
It remains to show that the even lift of $M_{16}$ does not fix $p$.
We will do this by showing that the cyclic subgroup of order $8$ within
is an odd lift.

An automorphism of order 8 of $B$ acts by $(t_0:t_1:t_2) \mapsto
(\epsilon_8^3t_0:\epsilon_8^{-1}t_1:t_2)$ in coordinates where $B$ is
given by the equation $t_2^4+t_0t_1(t_0^2+t_1^2) = 0$. It has 3 fixed
points in $\bbP^2$, two of which are on $B$. Thus, its lift must have 4
fixed points. The trace of an even lift of $g$ is equal to $-1$ and has
2 fixed points. We conclude that an element of order 8 in $G$ is an odd lift. 
Thus the subgroup isomorphic to $M_{16}$ in \ct{2B.5} is not minimal.
\end{proof}

\section{del Pezzo surfaces of degree $1$}
\begin{thm}[Case \ct{1}]
Let $X$ be a minimal del Pezzo $G$-surface of degree $1$.
Then $G$ has a fixed point.
\end{thm}

This is immediate since the unique base point of $|-K_X|$ is canonical,
and thus must be fixed by any automorphism of $X$.
A list of all the minimal groups in this case can be found in
Section~6.7~of~\cite{DI}.

\section{Conic Bundles}

Throughout this section,
$\pi : X \to B$ is a minimal conic $G$-bundle with $B \cong \bbP^1$.
Let $G_K$ be the kernel of the action of $G$ on $B$ and let
$G_B$ be the image.  We have an exact sequence
\[ 1 \to G_K \to G \to G_B \to 1 \ . \]
Let $\Sigma \subset B$ be the set of points whose preimages under $\pi$
are singular.

\begin{lem} \label{lem:GFfaithful}
The group $G_K$ acts faithfully on each fiber. 
\end{lem}

\begin{proof}
Suppose $g$ is a non-trivial element of $G_K$ that acts identically on a fiber $F$.
We know that $g$ has two fixed points on each nonsingular fiber.
The closure of this set of points is a one-dimensional component $C$ of
$X^g$ that is of relative degree 2 over the base.
More precisely, this curve is the closure of a divisor of degree 2 on
the general fiber $X_\eta$ that defines two fixed points of $G$ on the
geometric generic fiber.
Since $F$ and $C$ intersect and $X^g$ is smooth, we get a contradiction. 
\end{proof}

Let $G_0$ be the kernel of the action of $G$ on $\Pic(X)$.
We use the trichotomy of conic bundles as in \cite{DI}:
\begin{enumerate}
\item $X \to B$ is a ruled surface,
\item $X \to B$ is \emph{non-exceptional}: $G_0 = 1$ and $X$ is not ruled,
\item $X \to B$ is \emph{exceptional}: $G_0 \ne 1$ and $X$ is not ruled.
\end{enumerate}
We will consider each case in turn.

We begin by considering the case of a minimal ruled surface.
Recall that the case of $\mathbf{F}_0 \cong \bbP^1 \times \bbP^1$ was
shown to be birationally equivalent to $\bbP^2$ with a $G$-fixed point
in Section~\ref{sec:DP8}.
In fact, this is true for all ruled $G$-surfaces.

\begin{thm}
Suppose $X \cong \mathbf{F}_n$ is a ruled $G$-surface with a fixed point
where where $n \ge 2$.
Then $G$ is abelian and $X$ is birationally equivalent to
$\bbP^2$ with a $G$-fixed point.
\end{thm}

\begin{proof}
We recall some facts from the proof of Theorem 4.10 of~\cite{DI}.
Let $S$ be the exceptional section.
It is invariant with respect to the group of automorphisms of
$\bfF_n$.
The action of $G$ on $S$ is isomorphic to the action of $G$ on the base
of the projection $\pi:\bfF_n\to \bbP^1$.
Since $G$ has a fixed point on $X$, its projection is fixed, hence $G$
acts on $\bbP^1$ with two fixed points, and therefore $G$ has two fixed
points on $S$.
Since it has two fixed points on each invariant fiber, we obtain that
$G$ has 4 fixed points, two on $S$ and two outside $S$.

We now show that $X$ is birationally equivalent to $\bbP^2$ with a fixed
point.  Since $G$ fixes a point $p$ not on $S$, we may perform an elementary
transformation at that point to obtain a ruled $G$-surface $X'$
isomorphic to $\mathbf{F}_{n-1}$ which must also have a fixed point
since $G$ is abelian.  By applying this procedure inductively, we
eventually find a birational equivalence to a $G$-surface isomorphic to
$\mathbf{F}_1$.
Blowing down the exceptional divisor we have the desired result.
\end{proof}

\begin{thm}[Case \ct{C.ne}]
Let $G$ be a non-cyclic finite group and let $X$ be a
non-exceptional $G$-minimal conic bundle.
Assume that $G$ has a fixed point.
Then, we have one of the following cases:
\begin{itemize}
\item[(i)] $G_K \cong 2$, $G_B \cong 2n$, $G \cong 2\times 2n$,
\item[(ii)] $G_K \cong 2^2$, $G_B \cong n$, $G \cong 2\times 2n$,
\item[(iii)] $G_K \cong 2^2$, $G_B \cong n$, $G \cong (2m : 2)\times q$.
\end{itemize}
where $n=mq$ is a positive integer, $m$ is a power of $2$ and $q$
is odd.
\end{thm}

\begin{proof}
Recall that in this case $G_0$ is trivial and $X$ is not a ruled surface.
Here $G_K \cong 2^a$ with $a = 1$ or $2$ (\cite{DI}, Theorem 5.7).
Since $G$ has a fixed point, the group $G_B$ is cyclic.

First, consider $G_K \cong 2$.
Since we assume that $G$ is not cyclic, $G \cong 2\times 2n$ for some
positive integer $n$.

Assume now that $G_K \cong 2^2$.  The order of $G_B$ is $n=m q$
where $m$ is a power of $2$ and $q$ is a positive odd integer.
Thus there is a homomorphism from $G$ to a cyclic group of order $m$ whose
kernel is a $2$-group.  Since the quotient and kernel have coprime
orders, the extension splits.
It is known that the subgroup $G_K$ is always minimal (see \cite{DI},
Lemma 5.6).  The $2$-group will also be minimal since it contains
$G_K$.
Thus, it suffices to assume the order of $G_B$ is of the form $n=m$.

Since $G$ must embed into $\GL(2)$, we see that some element $z$ in
$G_K$ must map to the matrix $-\id$.  Let $x$ be a non-trivial element
of $G_K$ not equal to $z$.
Let $g$ be a lift to $G$ of a generator of $G_B$.

Since $g$ must normalize $G_K$ and $z$ must be central, we see that either
\begin{enumerate}
\item[(1)] $gxg^{-1}=x$, or
\item[(2)] $gxg^{-1}=xz$.
\end{enumerate}
In case (1), the group $G$ is abelian.  Since $G$ must have rank $\le
2$, we see that $G \simeq 2 \times 2m$.

In case (2), rearranging we obtain $xgx^{-1}=gz$.
Note that $xg^mx^{-1}=g^m$ since $m$ is even.
Thus, the group generated by
$\la x, z, g^m \ra$ must be abelian of rank $2$.
Since $G_K = \la x,z \ra$, and $x,g$ do not commute,
we conclude that either
\begin{enumerate}
\item[(a)] $g^m=1$, or
\item[(b)] $g^m=z$.
\end{enumerate}
In case (a), we conclude $g$ has order $2$ and $G \simeq D_8$;
otherwise, we would have a contradiction as the abelian group
$\la x, z, g^{m/2} \ra$ would have rank $3$.
In case (b), we conclude that our group $G$ is a semidirect product
$2m : 2$ where the involution $x$ acts by $g \mapsto g^{m+1}$.
\end{proof}

\begin{example}
Let $X$ be a del Pezzo surface of degree 4 and $G$ be a subgroup of
automorphisms generated by two involutions of the first kind, say
$\iota_1,\iota_2$.
The group has 4 fixed points $E_1\cap E_2 = \{p_1,p_2,p_3,p_4\}$.
Let $\sigma : X'\to X$ be the blowing up of $p_1$.
From the description of the involutions in Section~\ref{sec:DP4},
we wee that $p_1$ does not lie on any exceptional divisors.
Thus, the surface $X'$ is a del Pezzo surface of degree 3.
In its anti-canonical model, it is a cubic surface.

The image of the exceptional curve $E$ of $\sigma$ is a line on $X'$
invariant with respect to $G$.
The pencil of planes through $R$ has $R$ as its fixed component and the
residual pencil is a pencil of conics invariant under $G$.
It equips $X'$ with a structure of a minimal conic bundle $G$-surface
with a 2-section $R$.

Since $G \cong 2^2$ acts faithfully on the tangent space
of $X$ at $p_1$, and has two invariant tangent directions of $E_1$ and
$E_2$ at $p_1$, we see that the involution of the second kind
$\iota_{12}$ acts identically on $R$ while the other two involutions
act non-trivially.
We have a 2 to 1 morphism $R \cong \bbP^1 \to B \cong \bbP^1$
which is equiariant with respect to a cyclic group of order $2$;
this forces the action on $B$ to be trivial.
Thus the group $G_B$ is trivial and $G = G_K \cong 2^2$.

As $G$ acts faithfully on each fiber by Lemma~\ref{lem:GFfaithful},
we conclude that all of the fixed points must be the singular points of
the singular fibers of $\sigma$.
A conic bundle on a cubic surface has 5 singular fibers,
so we have 5 fixed points.  Alternatively, we note that there are
4 fixed points on $X$, but there are 2 fixed points on $R$;
thus $X'$ has 5 fixed points.

Note in the plane model of $X'$ as the blow-up of 6 points
$x_1,\ldots,x_5,p_1$, the pencil of conics arises from the pencil of
cubics through $x_1,\ldots,x_5$ and a double point at $p_1$. The
singular fibers are the unions of the line $\ell_i = \overline{p_1,x_i}$
and the conic $C_i$ through the points $x_k, k\ne i$ and $p_1$. Three of
such pairs $(l_i,C_i)$ intersect at $p_1$ and $p_i, i = 2,3,4$ and the
remaining two are tangent at $p_1$ with the tangent directions
corresponding to the cubics defined by $E_1$ and $E_2$.
\end{example}

\begin{thm}[Case \ct{C.ex}]
Let $G$ be a non-cyclic finite group and let $X$ be an
exceptional $G$-minimal conic bundle with a fixed point $p$.
Then $G_K$ is a dihedral group or a cyclic group of even order,
$G_B$ is cyclic or trivial,
and $G$ is a subgroup of $D_{2m} \times n$ for some
integers $m$ and $n$.
Furthermore, $p$ is a singular point of a singular fibre of $\pi$.
\end{thm}

\begin{proof}
Here $G_0$ is non-trivial, but $X$ is not a ruled surface.
Recall from Section~5~of~\cite{DI} that $X$ has 2
disjoint sections $S_0$ and $S_\infty$ that can be blown down to obtain
a hypersurface 
\[X': H_{2g+2}(t_0,t_1) + t_2t_3 = 0 \]
in the weighted projective space $\bbP(1:1\ ;\ g+1:g+1)$ for $g$ a
positive integer.
The map $\pi: X \to B$ is given by the morphism $(t_0:t_1\ ;\ t_2, t_3)
\to (t_0:t_1)$.

Since $G_B$ is cyclic or trivial, by the proof of
Proposition~5.3~of~\cite{DI} we see that $G$ is a subgroup of
$G_B \times N$ where $G_B$ acts on $(t_0:t_1)$ linearly and $N$ is the
subgroup $\bbC^\times : 2$ of $\SL_2(\bbC)$ which preserves $t_2t_3$.
Note that if $G$ is minimal then there must exist an element in $G_K$
which swaps $t_2$ and $t_3$ and thus has even order.
Since $G_K$ is a subgroup of even order of a dihedral group
it must be of the form in the statement of the theorem.

Finally, we establish that $G$ fixes a singular point of a singular
fiber.  
The subgroup $G_0$ leaves invariant each singular fiber and each section
$S_0$ and $S_\infty$.  Since there are $\ge 3$ singular fibers, the sections
$S_0$ and $S_\infty$ have $\ge 3$ points fixed by $G_0$.  Thus the
action of $G_0$ on $S_0$ and $S_\infty$ is trivial.  Since $G_0$ is a
subgroup of $G_K$, by Lemma~\ref{lem:GFfaithful}, it acts faithfully on
each fiber $F$.
In particular, it can only fix the points $F \cap S_0$ and $F \cap S_\infty$
on a non-singular fiber.
Since an element of $G_K$ swaps the two sections and $X^G \subset
X^{G_0}$ we see that $G$ can only fix the singular points of singular
fibers.
\end{proof}

\bibliographystyle{alpha}

\begin{thebibliography}{Dun13}


\bibitem[Atlas]{Atlas} J. Conway, R. Curtis, S.  Norton, R. Parker, R.
Wilson, \textit{Atlas of finite groups. Maximal subgroups and ordinary
characters for simple groups}. With computational assistance from J. G.
Thackray. Oxford University Press, Eynsham, 1985.

\bibitem[Bla07]{Blanc}
J.~Blanc.
\newblock Finite abelian subgroups of the {C}remona group of the plane.
\newblock {\em C. R. Math. Acad. Sci. Paris}, 344(1):21--26, 2007.

\bibitem[DO]{DO} I.~V.~Dolgachev and D.~Ortland.
\newblock Point sets in projective spaces and theta functions. 
\newblock {\em Ast\'erisque}, No. 165: 210 pp,  1988.

\bibitem[DI09]{DI}
I.~V.~Dolgachev and V.~A.~Iskovskikh.
\newblock Finite subgroups of the plane {C}remona group.
\newblock In {\em Algebra, arithmetic, and geometry: in honor of {Y}u. {I}.
  {M}anin. {V}ol. {I}}, volume 269 of {\em Progr. Math.}, pages 443--548.
  Birkh{\"a}user Boston Inc., Boston, MA, 2009.

\bibitem[Dol12]{CAG}
I.~V.~Dolgachev.
\newblock {\em Classical algebraic geometry}.
\newblock Cambridge University Press, Cambridge, 2012.

\bibitem[Dun13]{ed2}
A.~Duncan.
\newblock Finite groups of essential dimension 2.
\newblock {\em Comment. Math. Helv.}, 88(3):555--585, 2013.

\bibitem[Kow]{Kow} S.~Kowalevski, \newblock \"Uber Reduction  einer
bestimmten Klasse Abel'scher Integrale 3ten Ranges auf elliptische
Integrale.  \newblock {\em Acta Mathematica}, 4:393-416, 1884.

\bibitem[Pro13]{Prok}
Yu.~Prokhorov.
\newblock On stable conjugacy of finite subgroups of the plane {C}remona group,
  {II}, 2013.

\bibitem[RY00]{RY}
Z.~Reichstein and B.~Youssin.
\newblock Essential dimensions of algebraic groups and a resolution theorem for
  {$G$}-varieties.
\newblock {\em Canad. J. Math.}, 52(5):1018--1056, 2000.
\newblock With an appendix by J\'anos Koll\'ar and Endre Szab\'o.

\bibitem[Tsy11]{Tsygankov}
V.~I.~Tsygankov.
\newblock Equations of {$G$}-minimal conic bundles.
\newblock {\em Mat. Sb.}, 202(11):103--160, 2011.

\end{thebibliography}

\end{document}